\documentclass[a4paper]{scrartcl}
\usepackage{amsmath}
\usepackage{amssymb}
\usepackage{enumerate}
\usepackage{amsthm}
\usepackage{todonotes}
\usepackage[UKenglish]{babel}
\usepackage[T1]{fontenc}
\usepackage[utf8]{inputenc}
\usepackage{amsmath}
\usepackage{amssymb}
\usepackage{enumerate}
\usepackage{graphicx}

\newtheorem{theorem}{Theorem}[section]
\newtheorem{lemma}[theorem]{Lemma}
\newtheorem{corollary}[theorem]{Corollary} 
\theoremstyle{definition}
\newtheorem{definition}[theorem]{Definition}
\theoremstyle{remark}

\numberwithin{equation}{section}
\newcommand{\lip}{{\rm{Lip}}}
\newcommand{\dist}{{\rm{dist}}}

\newcommand{\N}{\mathbb{N}}
\newcommand{\R}{\mathbb{R}}
\newcommand{\inter}{{\rm{Int}}}
\newcommand{\diam}{{\rm{diam}}}

\newcommand{\K}{\ker}

\begin{document}

\title{On the structure of universal~differentiability~sets.\thanks{The research presented in this paper was completed when the author was a PhD student, supervised by Dr. Olga Maleva at University of Birmingham, UK. The paper is based on part of the author's PhD thesis. The author was supported by EPSRC funding.}}

\author{Michael Dymond}

\maketitle

\begin{abstract}
We prove that universal differentiability sets in Euclidean spaces possess distinctive structural properties. Namely, we show that any universal differentiability set contains a `kernel' in which the points of differentiability of each Lipschitz function are dense. We further prove that no universal differentiability set may be decomposed as a countable union of relatively closed, non-universal differentiability sets. The sharpness of this result, with respect to existing decomposibility results of the opposite nature, is discussed. 
\end{abstract}

\section{Introduction.}
Subsets of $\mathbb{R}^{d}$ containing a point of differentiability of every Lipschitz function $f:\mathbb{R}^{d}\to\mathbb{R}$ form a complex and still somewhat mysterious class of sets, despite significant modern progress. Such sets are called \emph{universal differentiability sets} (or UDSs), a term introduced in~\cite{doremaleva2}. The classical Rademacher's Theorem states that Lipschitz functions on Euclidean spaces are differentiable almost everywhere with respect to the Lebesgue measure. Thus, every set of positive measure is a universal differentiability set. Whilst one may characterise universal differentiability sets in $\mathbb{R}$ as sets of positive Lebesgue measure (see \cite{zahorski46}), this description fails in all Euclidean spaces of higher dimension. Preiss proves, in \cite{preiss90}, that $\mathbb{R}^{2}$ contains a dense, $G_{\delta}$ universal differentiability set of Lebesgue measure zero. In \cite{doremaleva1}, Dor\'e and Maleva verify the existence of compact universal differentiability sets of Lebesgue measure zero. Recent work in \cite{doremaleva2} and \cite{dymondmaleva13} establishes the existence of universal differentiability sets of an exceptional nature with respect to the Hausdorff and Minkowski dimensions. 

In order to better understand the nature of universal differentiability sets, it is helpful to study the class of porous and $\sigma$-porous sets. A subset $P\in\mathbb{R}^{d}$ is called porous if there exists $c\in(0,1)$ such that for every $x\in P$ and $\varepsilon>0$, there exists $h\in B(x,\varepsilon)$ such that $B(h,c\left\|h-x\right\|)\cap P=\emptyset$. A subset $F\subseteq\mathbb{R}^{d}$ is called $\sigma$-porous if $F$ may be expressed as a countable union of porous sets. $\sigma$-porous sets in $\mathbb{R}^{d}$ are negliglible in several senses. Any $\sigma$-porous set is of the first category with respect to the Baire Category Theorem. Moreover, unlike sets of the first category, $\sigma$-porous subsets of $\mathbb{R}^{n}$ must have Lebesgue measure zero, a consequence of the Lebesgue Density Theorem. For a survey on porous and $\sigma$-porous sets we refer to \cite{zajicek75}.

Porous and $\sigma$-porous sets are extremely relevant in the study of differentiability of Lipschitz functions because they admit Lipschitz functions which are nowhere differentiable inside them. Given a porous set $P\subseteq\mathbb{R}^{d}$ the distance function $x\mapsto \dist(x,P)$ is Lipschitz and fails to be differentiable at any point of $P$. Moreover, the paper \cite{preisszajicek01} proves that every $\sigma$-porous subset of $\mathbb{R}^{d}$ is a non-universal differentiability set. Universal differentiability sets are therefore necessarily non-$\sigma$-porous and further there appears to be an intimate relationship between these two classes of sets. We note, however, that non-$\sigma$-porosity is not sufficient for the universal differentiability property: Zajicek constructs in \cite{zajicek75} a non-$\sigma$-porous subset of the interval $[0,1]$ with Lebesgue measure zero. This is a non-universal differentiability set in $\mathbb{R}$ and its pre-image under the first co-ordinate projection map $p_{1}:\mathbb{R}^{d}\to\mathbb{R}$ gives a non-$\sigma$-porous, non-universal differentiability set in $\mathbb{R}^{d}$ with $d>1$.  

A fundamental property of non-$\sigma$-porous sets is that they cannot by decomposed: whenever a non-$\sigma$-porous set $G$ is expressed as a countable union of sets $G_{i}$, at least one of the sets $G_{i}$ has to be non-$\sigma$-porous. The research presented in this paper began with the question of whether universal differentiability sets exhibit the same robustness. That is, does every countable decomposition of a universal differentiability set necessarily contain a universal differentiability set? After a short time, we found an example to show that the answer is negative, using a result of Alberti, Cs\"ornyei and Preiss in \cite{acp05} (see Section~2). Nevertheless, it became clear that the universal differentiability property imposes strict demands on the nature of sets possessing it.

In this paper we examine the structural properties of universal differentiability sets in Euclidean spaces. Our approach is motivated by the work \cite{zelenypelant04}, of Zelen\'y and Pelant on the structure of non-$\sigma$-porous sets, which provides a level of insight into the nature of non-$\sigma$-porous sets not yet available for universal differentiability sets. In particular, we prove that, like non-$\sigma$-porous sets, universal differentiability sets contain a `kernel', which in some sense captures the core or essence of the set. In the papers \cite{doremaleva1}, \cite{doremaleva2} and \cite{doremaleva3}, Dor\'e and Maleva observe that the universal differentiability sets constructed possess the property that the differentiability points of each Lipschitz function form a dense subset. We verify that this is, broadly speaking, an intrinsic property of universal differentiability sets. We go on to establish that no universal differentiability set can be decomposed as a countable union of relatively closed, non-universal differentiability sets. Our main results are stated in Section~2 and proved in Section~3. Finally, in Section~4 we give an application to a question of Godefroy relating to the existence of exceptional universal~differentiability~sets.
\section{Main Results}
In this section we present our main results and discuss their connections to the existing theory. Our two main theorems are based on the following lemma:
\begin{lemma}\label{lemma:closed2}
	Let $F\subseteq\mathbb{R}^{d}$ be a universal differentiability set and suppose that $A$ is a relatively closed subset of $F$. Then either $A$ or $F\setminus A$ is a universal differentiability set.
\end{lemma}
In general, a universal differentiability set in $\mathbb{R}^{d}$ may be decomposed as a countable union of non-universal differentiability sets. Indeed, let $S$ be a universal differentiability set in $\mathbb{R}^{2}$ with Lebesgue measure zero (such a set is given in \cite{preiss90}). By a result of Alberti, Cs\"ornyei and Preiss in \cite{acp05}, there exist Lipschitz functions $f,g:\mathbb{R}^{2}\to\mathbb{R}$ such that $f$ and $g$ have no common points of differentiability inside $S$. Writing $D_{f}$ for the set of points of differentiability of $f$, we get that
\begin{equation*}
	S=(S\setminus D_{f})\cup (S\cap D_{f})
\end{equation*}
is a decomposition of $S$ as a union of two non-universal differentiabilty sets. 

The above discussion indicates that questions about decomposability of universal diffferentiability sets are closely related to questions about simultaneous differentiability of Lipschitz functions. The latter is an active area of current research: In \cite{lindenstrausstiserpreiss12}, Lindenstrauss, Preiss and Ti\v ser prove  that every pair of Lipschitz functions on a Hilbert space have a common point of differentiability; the corresponding question for a triple of Lipschtiz functions remains open. 

$G_{\delta}$ sets arise naturally in the theory of universal differentiability sets because they admit an equivalent metric with respect to which they are complete, see \cite{preiss90} and \cite{preissspeight2014}. We therefore considered the question of whether it is possible to weaken the assumption on $A$ in Lemma~\ref{lemma:closed2} to be $G_{\delta}$ rather than closed. It turns out that it is not possible to improve Lemma~\ref{lemma:closed2} in this manner: Cs\"ornyei, Preiss and Ti\v ser construct, in \cite{csornyeipreisstiser}, a universal differentiability set $S$ in $\mathbb{R}^{2}$ which may be decomposed as the union of two non-universal differentiability sets $U$ and $V$, where $U$ is uniformly purely unrectifiable. Since any uniformly purely unrectifiable set is both a non-universal differentiability set and contained in a $G_{\delta}$ uniformly purely unrectifiable set (see \cite{csornyeipreisstiser}), it follows that the set $U$ may be taken to be a $G_{\delta}$ subset of $S$.

We define the kernel, $\K(S)$, of a set $S\subseteq\mathbb{R}^{d}$ with respect to universal differentiability similarly to the kernel of $S$ with respect to non-$\sigma$-porosity, see \cite[Definition~3.2]{zelenypelant04}. 
\begin{definition}\label{def:kernel}
	\item Given a set $S\subseteq\mathbb{R}^{d}$, we let
	\begin{equation*}
		\K(S)=S\setminus\left\{x\in S\text{ : }\exists\varepsilon>0\text{ such that }B(x,\varepsilon)\cap S\text{ is a non-UDS}\right\}.
	\end{equation*}
\end{definition}
Note that $\ker(S)$ is closed as a subset of $S$. Through an application of Lemma~\ref{lemma:closed2} we obtain the following theorem, which shows that the kernel of a universal differentiability set can be thought of as the core of the set. We remark that universal differentiability sets behave similarly to non-$\sigma$-porous sets in this respect - see \cite[Lemma~3.4]{zelenypelant04}.
\begin{theorem}\label{theorem:kernel}
	Suppose $F\subseteq\mathbb{R}^{d}$ is a universal differentiability set. Then,
	\begin{enumerate}[(i)]
		\item $\K(F)\subseteq F$ is a universal differentiability set.
		\item $\K(\K(F))=\K(F)$ and $F\setminus \K(F)$ is a non-universal differentiability set. In particular, for each Lipschitz function $f:\mathbb{R}^{d}\to\mathbb{R}$, the differentiability points of $f$ in $\K(F)$ form a dense subset of $\K(F)$.
	\end{enumerate}
\end{theorem}
An iterative construction based on the proof of Lemma~\ref{lemma:closed2} leads to our second main theorem:
\begin{theorem}\label{theorem:fsigma}
	Suppose that $E\subseteq \mathbb{R}^{d}$ is a universal differentiability set and that $(A_{i})_{i=1}^{\infty}$ is a collection of relatively closed subsets of $E$ satisfying $E=\bigcup_{i=1}^{\infty}A_{i}$. Then at least one of the sets $A_{i}$ is a universal differentiability set.
\end{theorem}
We finish this section with a sketch of the proof of Lemma~\ref{lemma:closed2}. Suppose that the contrary holds for some universal differentiability set $F$ and a relatively closed, non-universal differentiability set $A\subseteq F$. We aim to construct a Lipschitz function $f:\mathbb{R}^{d}\to\mathbb{R}$ which is nowhere differentiable in $F$. This provides the desired contradiction. The set $F\setminus A$ is an open subset of $F$ and a non-universal differentiability set. Accordingly, we may cover $F\setminus A$ by a countable collection of boxes $U_{i}$ with pairwise disjoint interiors, such that $F\cap U_{i}$ is a non-universal differentiability set for each $i$. On each box $U_{i}$ we construct the function $f$ so that it is nowhere differentiable in $F\cap (U_{i}\setminus \partial U_{i})$ and satisfies $f=h$ on $\partial U_{i}$, where $h:\mathbb{R}^{d}\to \mathbb{R}$ is a fixed Lipschitz function nowhere differentiable on $A$. Outside of the boxes $U_{i}$ we simply set $f=h$. The obtained Lipschitz function $f$ is clearly nowhere differentiable in $F\cap(U_{i}\setminus\partial U_{i})$ for each $i$, whilst the non-differentiability of $f$ at all points of $A$ is achieved by controlling the size of the boxes $U_{i}$. Using the closedness of $A$, we ensure that the diameter of the boxes $U_{i}$ goes to zero with their distance from $A$. Therefore, around `difficult' points in $A$, the boundaries of the boxes $U_{i}$ become increasingly concentrated. Since $f=h$ on these boundaries, we get that the differentiability of $f$ coincides with that of $h$ at points of $A$.

\section{Construction.}
In this section we will prove the main results stated in Section 2. We begin with a summary of the notation that we will use: We fix an integer $d\geq 2$ and let $e_{1},e_{2},\ldots,e_{d}$ denote the standard basis of $\mathbb{R}^{d}$. For a point $x\in\mathbb{R}^{d}$ and $\varepsilon>0$, we let $B(x,\varepsilon)$ (respectively $\overline{B}(x,\varepsilon)$) denote the open (respectively closed) ball with centre $x$ and radius $\varepsilon$. The corresponding norm $\left\|-\right\|$ is the standard Euclidean norm on $\mathbb{R}^{d}$. Given a set $S\subseteq\mathbb{R}^{d}$, we let $\mbox{Int}(S)$ denote the interior, $\mbox{Clos}(S)$ denote the closure and $\partial{S}$ denote the boundary of $S$. For non-empty subsets $A$ and $B$ of $\mathbb{R}^{d}$ we let
\begin{equation*}
\diam(A)=\sup\left\{\left\|a'-a\right\|\mbox{ : }a,a'\in{A}\right\}\mbox{ and }\dist(A,B)=\inf\left\{\left\|b-a\right\|\mbox{ : }a\in{A},b\in{B}\right\}.
\end{equation*}
When $A=\left\{a\right\}$ is a singleton, we will just write $\mbox{dist}(a,B)$ rather than $\mbox{dist}(\left\{a\right\},B)$. We also adopt the convention $\mbox{dist}(A,\emptyset)=1$ for all $A\subseteq\mathbb{R}^{d}$. We write $\lip(f)$ for the Lipschitz constant of a Lipschitz function $f$. Moreover, given $e\in S^{d-1}$ and function $f\colon \R^{d}\to\R$ we let $f'(x,e):=\lim_{t\to 0+}\frac{f(x+te)-f(x)}{t}$ denote the one-sided directional derivative of $f$, provided that the limit exists. The restriction of $f$ to a set $S$ is denoted by $f|_{S}$ and the support of $f$ by $\mbox{supp}(f)$.

A subset $U$ of $\mathbb{R}^{d}$ is called a box if $U=I_{1}\times I_{2}\times \ldots\times I_{d}$ for some sequence of closed, bounded intervals $I_{1},\ldots,I_{d}\subseteq\mathbb{R}$. Writing $I_{k}=[a_{k},b_{k}]$ for each $k$, we call a set $Y\subseteq\partial U$ a face of $U$ if there exist $m\in\left\{1,\ldots,d\right\}$ and $y\in\left\{a_{m},b_{m}\right\}$ such that $Y=I_{1}\times\ldots\times I_{m-1}\times\left\{y\right\}\times I_{m+1}\times\ldots\times I_{d}$. Note that each face of $U$ is a subset of a $(d-1)$-dimensional hyperplane which is orthogonal to exactly one of the vectors $e_{1},\ldots,e_{d}$.

\begin{lemma}\label{lemma:respectcover}
	Suppose that $A$ is a relatively closed subset of $F\subseteq\mathbb{R}^{d}$. Then, there exists a collection $\left\{U_{i}\right\}_{i=1}^{\infty}$ of boxes with pairwise disjoint interiors such that 
	\begin{equation*}
		F\cap\bigcup_{i=1}^{\infty}U_{i}=F\setminus A\quad\text{ and }\quad\frac{\diam(U_{i})}{\dist(U_{i},A)}\to 0\mbox{ as }i\to\infty.
	\end{equation*}
	Furthermore, if $\left\{V_{i}\right\}_{i=1}^{\infty}$ is a collection of boxes with pairwise disjoint interiors satisfying $F\setminus A\subseteq F\cap\bigcup_{i=1}^{\infty}V_{i}$, then the collection $\left\{U_{i}\right\}_{i=1}^{\infty}$ can be chosen so that for each index $i$ there exists an index $j$ such that $U_{i}\subseteq{V_{j}}$.
\end{lemma}
\begin{proof}
	Let $\left\{S_{i}\right\}_{i=1}^{\infty}$ be a collection of boxes with pairwise disjoint interiors such that $F\cap\bigcup_{i=1}^{\infty}S_{i}=F\setminus A$. Such a collection is easy to construct using the fact that $A$ is relatively closed in $F$. For each $i,k\geq{1}$ we let $S_{i,k}=S_{i}\cap V_{k}$ and note that $F\cap\bigcup_{i,k=1}^{\infty}S_{i,k}=F\setminus A$. Moreover, $\mathcal{W}=\left\{S_{i,k}\right\}_{i,k\geq{1}}$ is a countable collection of boxes with pairwise disjoint interiors. After relabelling, we can write $\mathcal{W}=\left\{W_{i}\right\}_{i=1}^{\infty}$. 
	
	Set $p_{0}=0$. For each $i\geq{1}$, partition the box $W_{i}$ into a finite number of boxes $U_{p_{i-1}+1},\ldots,U_{p_{i}}$ with pairwise disjoint interiors such that
	\begin{equation*}
		\frac{\diam(U_{j})}{\dist(U_{j},A)}\leq 2^{-i}\mbox{ for }p_{i-1}+1\leq j\leq p_{i}.
	\end{equation*}
	The assertions of the Lemma are now readily verified.
\end{proof}
The proof of Theorem~\ref{theorem:fsigma} will involve constructing a sequence of Lipschitz functions which converge to a Lipschitz function nowhere differentiable on the set $E$. We will need the next Lemma in order to pick, at each step of the construction, a Lipschitz function nowhere differentiable on part of the set.
\begin{lemma}\label{obs:nearzero}
	Let $\Omega$ be a bounded subset of $\mathbb{R}^{d}$. A set $E\subseteq\mathbb{R}^{d}$ is a non-universal differentiability set if and only if for every $\varepsilon>0$, there exists a Lipschitz function $f\colon \R^{d}\to\R$ such that $\left\|f|_{\Omega}\right\|_{\infty}\leq\varepsilon$, $\lip(f)\leq\varepsilon$ and $f$ is nowhere differentiable in $E$.
\end{lemma}
\begin{proof}
	We focus only on the non-trivial direction. Suppose that there exists $\varepsilon>0$ such that every Lipschitz function $f\colon \R^{d}\to\R$, with $\left\|f_{\Omega}\right\|_{\infty}$ and $\lip(f)\leq\varepsilon$, has a point of differentiability inside $E$. We may assume that $\varepsilon<1$.
	
	Fix a Lipschitz function $g\colon \R^{d}\to\R$. We show that $g$ has a point of differentiability inside $E$. We note that $\left\|g|_{\Omega}\right\|_{\infty}+\lip(g)$ is finite, since $\Omega$ is bounded. Further, we may assume that $\left\|g|_{\Omega}\right\|_{\infty}+\lip(g)\neq{0}$.  Let
	\begin{equation*}
		h(x)=\frac{\varepsilon}{2}+\frac{\varepsilon}{4(\left\|g|_{\Omega}\right\|_{\infty}+\lip(g))}g(x)\quad\text{ for all }x\in\mathbb{R}^{d}.
	\end{equation*}
	Note that $\left\|h(x)-\frac{\varepsilon}{2}\right\|\leq\frac{\varepsilon}{4}$ for all $x\in\Omega$. Hence $\left\|h|_{\Omega}\right\|_{\infty}<\varepsilon$. Moreover, we have $h\colon \R^{d}\to\R$ with $\mbox{Lip}(h)\leq\frac{\varepsilon}{4}<\varepsilon$. Hence, there exists a point $x\in{E}$ such that $h$ is differentiable at $x$. But then  $g$ is also differentiable at $x$ because $g=\alpha h+\beta$ where $\alpha,\beta\in\mathbb{R}$ are fixed constants. Since $g\colon \R^{d}\to\R$ was arbitrary, we deduce that $E$ is a universal differentiability set.
\end{proof}
Using the next Lemma, we will later be able to ignore points lying in the boundaries of boxes.
\begin{lemma}\label{lemma:removingboundaries}
	Suppose $E$ is a subset of $\mathbb{R}^{d}$ and $\left\{U_{i}\right\}_{i=1}^{\infty}$ is a collection of boxes in $\mathbb{R}^{d}$ such that $E\setminus\bigcup_{i=1}^{\infty}\partial U_{i}$ is a non-universal differentiability set. Then $E$ is a non-universal differentiability set.
\end{lemma}
\begin{proof}
	For $j=1,\ldots,d$, let $H_{j}$ denote the union of all faces of the boxes $U_{i}$ which are orthogonal to $e_{j}$ Then $\bigcup_{i=1}^{\infty}\partial U_{i}=\bigcup_{j=1}^{d}H_{j}$. Moreover, writing $p_{j}$ for the $j$-th co-ordinate projection map on $\mathbb{R}^{d}$, we have that $p_{j}(H_{j})$ is a subset of $\mathbb{R}$ with one-dimensional Lebesgue measure zero, (in fact it is a countable set) for each $j=1,\ldots,d$. The result now follows from \cite[Lemma~2.1]{dymondmaleva13}.
\end{proof}

\begin{lemma}\label{lemma:main}
	Let $E\subseteq\mathbb{R}^{d}$, $\eta>0$ and $\left\{U_{i}\right\}_{i=1}^{\infty}$ be a collection of boxes with pairwise disjoint interiors such that and $E\cap \inter(U_{i})$ is a non-universal differentiability set for each $i$. Then there exists a function $g\colon \R^{d}\to\R$ such that
	\begin{align}
		&\label{smalllipnorm} \left\|g\right\|_{\infty}\leq\eta\text{ and }\lip(g)\leq\eta,\\
		&\label{notdiffint} g\text{ is nowhere differentiable in }\displaystyle\bigcup_{i=1}^{\infty}E\cap\mbox{Int}(U_{i}),\text{ and}\\
		&\label{eq:zerobdy} g(x)=0\quad\text{ whenever }x\in\mathbb{R}^{d}\setminus\left(\displaystyle\bigcup_{i=1}^{\infty}\mbox{Int}(U_{i})\right).
	\end{align}
\end{lemma}
\begin{proof}
	For $x\in\R^{d}\setminus\left(\bigcup_{i=1}^{\infty}\inter(U_{i})\right)$ we define $g(x)$ according to \eqref{eq:zerobdy}. Given $i\in\N$ we define $g$ on $\inter(U_{i})$ as follows: Let $\varphi=\varphi_{i}\in C^{\infty}(\mathbb{R}^{d})$ be a smooth function satisfying $\varphi(x)>0$ for all $x\in\inter(U_{i})$, $\left\|\varphi\right\|_{\infty}\leq 1$, $\lip(\varphi)\leq 1$ and $\varphi(x)=0$ for all $x\in\mathbb{R}^{d}\setminus\inter(U_{i})$. By Lemma \ref{obs:nearzero}, there exists a Lipschitz function $h=h_{i}\colon \R^{d}\to\R$ such that $h$ is nowhere differentiable in $E\cap\mbox{Int}(U_{i})$, $\left\|h|_{U_{i}}\right\|_{\infty}\leq\eta/2$ and $\lip(h)\leq\eta/2$. We define $g$ on $\inter(U_{i})$ by $g(x)=\varphi(x)h(x)$. The smoothness of $\varphi$ and the fact that $\varphi>0$ on $\inter(U_{i})$ ensure that $g$ inherits all non-differentiability points of $h$ in $\inter(U_{i})$. The assertions of the lemma are now readily verified. 
\end{proof}
The previous two lemmas admit the following corollary:
\begin{corollary}\label{cor:nudschar}
	A set $E\subseteq \mathbb{R}^{d}$ is a non-universal differentiability set if and only if for every $x\in E$, there exists $\varepsilon=\varepsilon_{x}>0$ such that $B(x,\varepsilon)\cap E$ is a non-universal differentiability set.
\end{corollary}
In the next lemma, we show that, given a Lipschitz function $h$ and a collection of pairwise disjoint boxes, we can slightly modify the function $h$ so that it becomes differentiable everywhere in the interiors of the boxes and remains unchanged everywhere else.
\begin{lemma}\label{lemma2}
	Let $\left\{U_{i}\right\}_{i=1}^{\infty}$ be a collection of boxes in $\mathbb{R}^{d}$ with pairwise disjoint interiors, $h\colon \R^{d}\to\R$ be a Lipschitz function and $\sigma>0$. Then, there exists a Lipschitz~function $\widehat{h}\colon \R^{d}\to\R$ such that 
	\begin{align}
		&\left\|\widehat{h}-h\right\|_{\infty}\leq\sigma\text{ and }\lip(\widehat{h})\leq\lip(h)+\sigma,\label{eq:lipnorm2}\\
		&\widehat{h}\text{ is everywhere differentiable inside }\displaystyle\bigcup_{i=1}^{\infty}\inter(U_{i})\text{, and}\label{eq:evwdiff2}\\
		&\widehat{h}(x)=h(x)\quad\text{ for all }x\in\mathbb{R}^{d}\setminus\displaystyle\bigcup_{i=1}^{\infty}{\inter(U_{i})}\label{eq:borders2}.
	\end{align}
\end{lemma}
\begin{proof}
	We may assume that $\lip(h)>0$. Outside of $\bigcup_{i=1}^{\infty}\inter(U_{i})$, we define the function $\widehat{h}$ according to \eqref{eq:borders2}. Given $i\in\N$ we define $\widehat{h}$ on $\inter(U_{i})$ as follows: Fix a $C^{\infty}$ function $\pi:\mathbb{R}^{d}\to\mathbb{R}$ such that $\pi(z)\geq{0}$ for all $z\in\mathbb{R}^{d}$, $\text{supp}(\pi)\subseteq{B(0,1)}$ and $\int_{\mathbb{R}^{d}}\pi(z)dz=1$. Choose a $C^{\infty}$ function $\gamma=\gamma_{i}:\R^{d}\to\mathbb{R}$ such that $0<\gamma(x)\leq\sigma/\lip(h)$ for all $x\in\inter(U_{i})$, $\gamma(x)=0$ for all $x\in\R^{d}\setminus \inter(U_{i})$ and $\lip(\gamma)\leq\sigma/\lip(h)$. We now define $\widehat{h}$ on $\inter(U_{i})$ by $\widehat{h}(x)=\frac{1}{\gamma(x)^{d}}\int_{\mathbb{R}^{d}}h(x-z)\pi(z/\gamma(x))dz$ for all $x\in\inter (U_{i})$.
	
	We presently verify statement \eqref{eq:lipnorm2} for $\widehat{h}$. Fix $i\in\N$ and $x\in\inter(U_{i})$. Then
	\begin{align*}
	\left\|\widehat{h}(x)-h(x)\right\|&=\left\|\frac{1}{\gamma(x)^{d}}\int_{\mathbb{R}^{d}}h(x-z)\pi(z/\gamma(x))dz-h(x)\frac{1}{\gamma(x)^{d}}\int_{\mathbb{R}^{d}}\pi(z/\gamma(x))dz\right\|\\
	&\leq \frac{1}{\gamma(x)^{d}}\int_{\mathbb{R}^{d}}\left\|h(x-z)-h(x)\right\|\pi(z/\gamma(x))dz
	\leq\gamma(x)\lip(h)\leq\sigma.
	\end{align*}
	This establishes the first part of \eqref{eq:lipnorm2}. The penultimate inequality also proves that $h$ is continuous on $\R^{d}$, since $\gamma(x)=0$ for all $x\in\partial U_{i}$. Thus, for the second part of \eqref{eq:lipnorm2}, it suffices to check the Lipschitz condition for points $x,y\in\inter(U_{i})$. For such $x,y$ we have
	\begin{align*}
	\left\|\widehat{h}(y)-\widehat{h}(x)\right\|&=\left\|\frac{1}{\gamma(y)^{d}}\int_{\mathbb{R}^{d}}h(y-z)\pi(z/\gamma(y))dz-\frac{1}{\gamma(x)^{d}}\int_{\mathbb{R}^{d}}h(x-z)\pi(z/\gamma(x))dz\right\|\\
	&\leq\int_{B(0,1)}\left\|h(y-\gamma(y)v)-h(x-\gamma(x)v)\right\|\pi(v)dv\\
	&\leq \lip(h)(1+\lip(\gamma))\left\|y-x\right\|\leq (\lip(h)+\sigma)\left\|y-x\right\|,
	\end{align*}
	as required. Finally, let us verify \eqref{eq:evwdiff2}. 
	Fix $i\in\N$ and define for each $u\in\mathbb{R}^{d}$ a map $\Phi_{u}:\inter(U_{i})\to \mathbb{R}^{d}$ by 
	\begin{equation*}
	\Phi_{u}(x)=\pi(x-u/\gamma(x)).
	\end{equation*}
	Note that each $\Phi_{u}$ is a differentiable function and for $x\in \inter(U_{i})$ we can write $\widehat{h}(x)=\widetilde{h}(x)/\gamma(x)^{d}$ where
	\begin{equation*}
	\widetilde{h}(x)=\int_{\mathbb{R}^{d}}h(u)\Phi_{u}(x)du
	\end{equation*}
	Since the function $\inter(U_{i})\to\mathbb{R}$, $x\mapsto 1/\gamma(x)^{d}$ is everywhere differentiable, it suffices to prove that $\widetilde{h}$ is every differentiable in $\inter(U_{i})$. We show that
	\begin{equation*}
	\widetilde{h}'(x,e)=\int_{\mathbb{R}^{d}}h(u)\Phi_{u}'(x,e)du\quad\text{ for all }x\in\inter(U_{i}),e\in S^{d-1}.
	\end{equation*}
	
	Fix $x\in\inter(U_{i})$ and $e\in S^{d-1}$. Next, choose $t_{0}$ small enough so that for all $t\in(0,t_{0})$ we have 
	\begin{equation}\label{eq:t0}
	x+te\in\inter(U_{i}),\qquad 1/2\leq \gamma(x+te)/\gamma(x)\leq 2,\qquad t/\gamma(x+te)\leq 1/2.
	\end{equation}
	Observe that
	\begin{equation*}
	\frac{\widetilde{h}(x+te)-\widetilde{h}(x)}{t}=\int_{\mathbb{R}^{d}}h(u)\left[\frac{\Phi_{u}(x+te)-\Phi_{u}(x)}{t}\right]du,
	\end{equation*}
	and for each $u$ we have
	\begin{equation*}
	\lim_{t\to 0}h(u)\left[\frac{\Phi_{u}(x+te)-\Phi_{u}(x)}{t}\right]=h(u)\Phi_{u}'(x,e).
	\end{equation*}
	Therefore, using the Dominated Convergence Theorem, we only need to show that there exists $\varphi\in L^{1}(\mathbb{R}^{d})$ such that 
	\begin{equation*}\label{eq:finite}
	\left\|h(u)\left[\frac{\Phi_{u}(x+te)-\Phi_{u}(x)}{t}\right]\right\|\leq \varphi(u)\quad\text{ for all }\quad t\in(0,t_{0}),u\in\mathbb{R}^{d}.
	\end{equation*}
	Let $t\in(0,t_{0})$. Suppose $\left\|u-x\right\|>4\gamma(x)$. Then, from \eqref{eq:t0}, 
	\begin{equation*}
	\left\|\frac{x+te-u}{\gamma(x+te)}\right\|\geq\frac{4\gamma(x)}{\gamma(x+te)}-\frac{t}{\gamma(x+te)}\geq 2-\frac{1}{2}>1.
	\end{equation*}
	It follows that $\pi(x+te-u/\gamma(x+te))=0$. Similarly, we have $\pi(x-u/\gamma(x))=0$. Hence, $\Phi_{u}(x+te)-\Phi_{u}(x)/t=0$ for all $u\in \R^{d}\setminus B(0,4\gamma(x))$ and $t\in(0,t_{0})$. 
	
	It now suffices to show that $\left\|h(u)\left[\frac{\Phi_{u}(x+te)-\Phi_{u}(x)}{t}\right]\right\|$ is uniformly bounded for $u\in B(x,4\gamma(x))$ and $t\in(0,t_{0})$. Fix $u\in B(x,4\gamma(x))$ and $t\in(0,t_{0})$. Then, using \eqref{eq:t0} we obtain
	\begin{align*}
	&\left\|h(u)\left[\frac{\Phi_{u}(x+te)-\Phi_{u}(x)}{t}\right]\right\|\leq\left\|h\right\|_{\infty}\lip(\pi)\left\|\frac{x+te-u}{t\gamma(x+te)}-\frac{x-u}{t\gamma(x)}\right\|\\
	\leq& \left\|h|_{\overline{B}(x,4\gamma(x))}\right\|_{\infty}\lip(\pi)\frac{2}{\gamma(x)^{2}}\left(\left\|\frac{\gamma(x+te)-\gamma(x)}{t}\right\|\left\|x-u\right\|+\gamma(x)\right)\\
	\leq& \left\|h|_{\overline{B}(x,4\gamma(x))}\right\|_{\infty}\lip(\pi)\frac{2}{\gamma(x)^{2}}(\lip(\gamma)4\gamma(x)+\gamma(x)).
	\end{align*}
\end{proof}
In our main construction, we are faced with a situation where we would like to slightly modify a Lipschitz function $h$ to obtain a new function $f$, whilst preserving non-differentiability points. Using the previous lemma, we ensure that $f$ coincides with $h$ on the boundaries of boxes. The application of the following lemma, is to show that if these boundaries become increasingly concentrated around a point $x$, then the differentiability of $f$ and $h$ at $x$ coincide.
\begin{lemma}\label{lemma:dirdiff=}
	Let $x\in\mathbb{R}^{d}$, $e\in S^{d-1}$, $f,h\colon \R^{d}\to\R$ be Lipschitz functions and $\Delta>0$. Suppose that $\left\{(t_{k,1},t_{k,2})\right\}_{k=1}^{N}$, where $0<t_{k,1}\leq t_{k,2}$ and $N\in\N\cup\left\{\infty\right\}$, is a finite or countable collection of open, possibly degenerate intervals inside $(0,\infty)$ such that the following conditions hold:
	\begin{align}
		&\text{If $N=\infty$ then }\frac{t_{k,2}-t_{k,1}}{t_{k,1}}\to 0\text{ as }k\to\infty.\label{eq:shrink}\\
		&f(x+te)=h(x+te)\qquad \forall t\in\left([0,\Delta)\setminus\bigcup_{k=1}^{N}(t_{k,1},t_{k,2})\right)\cup\left(\bigcup_{k=1}^{N}\left\{t_{k,1},t_{k,2}\right\}\right)\label{eq:agree}.
	\end{align}
	Then the directional derivative $f'(x,e)$ exists if and only if the directional derivative~$h'(x,e)$ exists and $f'(x,e)=h'(x,e)$.
\end{lemma}
\begin{proof}
	Note that the statement is symmetric with respect to $f$ and $h$. We may assume that $\lip(f)+\lip(h)>0$. 
	
	Suppose that the directional derivative $f'(x,e)$ exists. Fix $\varepsilon>0$ and choose $\delta>0$ so that
	\begin{equation}\label{eq:delta}
	\left|f(x+te)-f(x)-tf'(x,e)\right|\leq\frac{\varepsilon}{3}t\quad\text{whenever}\quad0<t<\delta.
	\end{equation}
	If $N<\infty$ we set $K=N$. Otherwise, using \eqref{eq:shrink}, we may pick $K\geq{1}$ large enough so that
	\begin{equation}\label{eq:K}
	\frac{t_{k,2}-t_{k,1}}{t_{k,1}}\leq\frac{\varepsilon}{3(\lip(f)+\lip(h))}\quad\text{for every}\quad k\geq K.
	\end{equation}
	Let $t\in(0,\min\left\{\delta,\Delta,t_{1,1},\ldots,t_{K,1}\right\}$. We distinguish two cases:\\
	First assume that $t\in(0,\infty)\setminus\left(\bigcup_{k=1}^{N}(t_{k,1},t_{k,2})\right)$. Then, combining \eqref{eq:agree}, \eqref{eq:delta} and $t\in(0,\min\left\{\delta,\Delta\right\})$ we get
	\begin{equation*}
		\left|h(x+te)-h(x)-tf'(x,e)\right|=\left|f(x+te)-f(x)-tf'(x,e)\right|\leq\frac{\varepsilon}{3}t.
	\end{equation*}
	In the remaining case there exists $k$ such that $t\in(t_{k,1},t_{k,2})$. Moreover, since $t<t_{l,1}$ for $1\leq l\leq K$, we must have $k>K$. Note, in particular, that in this case $N=\infty$. Using \eqref{eq:agree}, \eqref{eq:K} and \eqref{eq:delta} we deduce
	\begin{align*}
		&\left|h(x+te)-h(x)-tf'(x,e)\right|\leq\left|h(x+te)-h(x+t_{k,1}e)\right|\\&+\left|f(x+t_{k,1}e)-f(x+te)\right|+\left|f(x+te)-f(x)-tf'(x,e)\right|\\
		&\leq(\lip(h)+\lip(f))(t-t_{k,1})+\frac{\varepsilon}{3}t\leq\frac{(\lip(h)+\lip(f))(t_{k,2}-t_{k,1})}{t_{k,1}}t+\frac{\varepsilon}{3}{t}\leq\varepsilon{t}.
	\end{align*}
	We have now established that the directional derivative $h'(x,e)$ exists and equals $f'(x,e)$.
\end{proof}

\begin{lemma}\label{lemma:closed}
	Let $F\subseteq\mathbb{R}^{d}$ and let $A$ be a closed subset of $F$. Let $h\colon \R^{d}\to\R$ be a Lipschitz~function and let $\left\{U_{i}\right\}_{i=1}^{\infty}$ be a collection of boxes with pairwise disjoint interiors such that
	\begin{equation}\label{eq:tiling}
	F\cap\bigcup_{i=1}^{\infty}U_{i}=F\setminus A\quad\text{and}\quad\frac{\diam(U_{i})}{\dist(U_{i},A)}\to 0\quad\text{as}\quad i\to\infty.
	\end{equation}
	Then the following two statements hold:
	\begin{enumerate}
		\item Suppose that $F\setminus A$ is a non-universal differentiability set and let $\varepsilon>0$. Then there exists a Lipschitz function $f\colon \R^{d}\to\R$ with the following properties:
		\begin{align}
			&\text{$f$ is nowhere differentiable in}\quad (F\setminus A)\setminus\left(\bigcup_{i=1}^{\infty}\partial U_{i}\right),\label{eq:nowherediff}\\
			&\left\|f-h\right\|_{\infty}\leq\varepsilon\quad\text{and}\quad\lip(f)\leq\lip(h)+\varepsilon,\label{eq:close}\\
			&f(y)=h(y)\quad\text{whenever}\quad y\in \R^{d}\setminus\bigcup_{i=1}^{\infty}\inter(U_{i}).\label{eq:coincide}
		\end{align}
		\item Let $x\in A$ and suppose $f\colon \R^{d}\to\R$ is any Lipschitz function satisfying the condition \eqref{eq:coincide}. Then $f$ is differentiable at $x$ if and only if $h$ is differentiable at $x$.
	\end{enumerate}
\end{lemma}
\begin{proof}
	Let us first verify statement 1. Suppose $F\setminus A$ is a non-universal differentiability set and let $\varepsilon>0$. Note that the boxes $\left\{U_{i}\right\}_{i=1}^{\infty}$, the function $h$ and $\sigma=\varepsilon/2$ satisfy the conditions of Lemma~\ref{lemma2}. Let $\widehat{h}\colon \R^{d}\to\R$ be the function given by the conclusion of Lemma~\ref{lemma2}.
	
	The conditions of Lemma~\ref{lemma:main} are satisfied for $E=F\setminus A$, the collection $\left\{U_{i}\right\}_{i=1}^{\infty}$ and $\eta=\varepsilon/2$. Let $g\colon \R^{d}\to\R$ be given by the conclusion of Lemma~\ref{lemma:main}. We define $f=\widetilde{h}+g$. The assertions \eqref{eq:nowherediff}, \eqref{eq:close} and \eqref{eq:coincide} now follow easily from the properties of $\widetilde{h}$ and $g$.
	
	Finally, we prove statement 2. Suppose $f\colon \R^{d}\to\R$ satisfies \eqref{eq:coincide} and $x\in A$. Let $e\in S^{d-1}$ be any direction. We show that $f'(x,e)$ exists if and only if $h'(x,e)$ exists and $f'(x,e)=h'(x,e)$. Since $e\in S^{d-1}$ is arbitrary, this suffices.
	
	Let $\left\{U_{i_{k}}\right\}_{k=1}^{N}$, where $N\in\N\cup\left\{\infty\right\}$, be the collection of all boxes $U_{i}$ which intersect the line $x+[0,\infty)e$. Since $x\in A\subseteq F$ and \eqref{eq:tiling} holds, we have that $x\notin\bigcup_{i=1}^{\infty}U_{i}$. For each $k\geq 1$ we write
	\begin{equation}\label{eq:bdypoints}
	U_{i_{k}}\cap\left(x+(0,\infty)e\right)=\left[x+t_{k,1}e,x+t_{k,2}e\right]
	\end{equation}
	where $t_{k,1}\leq t_{k,2}$ are strictly positive real numbers. Choose $\Delta>0$ small enough so that $B(x,\Delta)\subseteq O$. We note that the following condition is satisfied:
	\begin{equation}\label{eq:x+te}
	x+te\in\begin{cases}
	\partial U_{i_{k}} & \text{if }t=t_{k,j},\,j=1,2,\\
	\R^{d}\setminus\left(\bigcup_{i=1}^{\infty}\inter(U_{i})\right) & \text{if }t\in \left([0,\Delta)\setminus\bigcup_{k=1}^{N}(t_{k,1},t_{k,2})\right)\cup\bigcup_{k=1}^{N}\left\{t_{k,1},t_{k,2}\right\}. 
	\end{cases}
	\end{equation}

	Let us verify that the conditions of Lemma~\ref{lemma:dirdiff=} hold for $x$, $e$, $f$, $h$, $\Delta$ and the intervals $\left\{(t_{k,1},t_{k,2})\right\}_{k=1}^{N}$. In view of the conclusion of Lemma~\ref{lemma:dirdiff=}, this completes the proof.
	
	If $N=\infty$, then using \eqref{eq:bdypoints}, we get that
	\begin{equation*}
		0\leq\frac{t_{k,2}-t_{k,1}}{t_{k,1}}\leq\frac{\diam(U_{i_{k}})}{\dist(U_{i_{k}},A)}\to0\quad\text{as}\quad k\to\infty.
	\end{equation*}
	This proves \eqref{eq:shrink}. Finally, we note that \eqref{eq:agree} follows from \eqref{eq:x+te} and \eqref{eq:coincide}.
\end{proof}
We are now ready to combine the results of the present section in proofs of our main results:
\begin{proof}[Proof of Lemma~\ref{lemma:closed2}]
	Suppose the contrary for some $A\subseteq F\subseteq\mathbb{R}^{d}$. Then there exists a Lipschitz function $h\colon \R^{d}\to\R$ such that $h$ is nowhere differentiable in $A$. Applying Lemma~\ref{lemma:respectcover}, we find a collection $\left\{U_{i}\right\}_{i=1}^{\infty}$ of boxes with pairwise disjoint interiors such that \eqref{eq:tiling} holds. The conditions of Lemma~\ref{lemma:closed} are satisfied for $F$, $A$, $h$ and $\left\{U_{i}\right\}_{i=1}^{\infty}$. Further, the hypothesis of Lemma~\ref{lemma:closed} part~1 is satisfied for $F$, $A$ and arbitrary $\varepsilon>0$. Let $f\colon \R^{d}\to\R$ be the Lipschitz function given by the conclusion of Lemma~\ref{lemma:closed}, part 1.
	
	By Lemma~\ref{lemma:closed}, part 2 the differentiability of $f$ and $h$ coincides at all points of $A$. Thus $f$ is nowhere differentiable in $A$. Moreover, $f$ is nowhere differentiable in $(F\setminus A)\setminus\bigcup_{i=1}^{\infty}\partial U_{i}$ by \eqref{eq:nowherediff}. Hence $F\setminus\bigcup_{i=1}^{\infty}\partial U_{i}$ is a non-universal differentiability set and Lemma~\ref{lemma:removingboundaries} asserts that $F$ is also a non-universal differentiability set.
\end{proof}

\begin{proof}[Proof of Theorem~\ref{theorem:kernel}]
	Note that $\K(F)$ is a closed subset of $F$ and $F\setminus \K(F)$ is a non-universal differentiability set by Corollary \ref{cor:nudschar}. Therefore, we may apply Lemma~\ref{lemma:closed2} with $A=\K(F)$ to deduce that $\K(F)$ is a universal differentiability set. This proves (i). For (ii), it only remains to check that $\K(\K(F))=\K(F)$. Let $x\in \K(F)$ and $\varepsilon>0$. Then we observe that
	\begin{equation*}
		B(x,\varepsilon)\cap \K(F)=\K(B(x,\varepsilon)\cap F),
	\end{equation*}
	and the latter set is a universal differentiability set by part (i) and $x\in\K(F)$.
\end{proof}

\begin{proof}[Proof of Theorem~\ref{theorem:fsigma}]
	Suppose that the contrary holds for some universal differentiability set $E\subseteq\mathbb{R}^{d}$. This means that there exist relatively closed subsets $A_{i}$ of $E$ such that $E=\bigcup_{i=1}^{\infty}A_{i}$ and each $A_{i}$ is a non-universal differentiability set. By Lemma~\ref{lemma:closed2}, we may assume that $A_{k}\subseteq{A_{k+1}}$ for each $k\geq{1}$.  We will obtain a contradiction, by proving that $E$ is a non-universal differentiability set.
	
	We begin the construction by using Lemma~\ref{lemma:respectcover} to find a collection of boxes $\left\{U_{i,1}\right\}_{i=1}^{\infty}$ with pairwise disjoint interiors such that 
	\begin{equation*}
	E\cap\bigcup_{i=1}^{\infty}U_{i,1}=E\setminus A_{1}\quad\text{ and }\quad\frac{\diam(U_{i,1})}{\dist(U_{i,1},A_{1})}\to 0\text{ as }i\to\infty.
	\end{equation*}
	Choose a Lipschitz function $f_{1}\colon \R^{d}\to\R$ such that $f_{1}$ is nowhere differentiable in $A_{1}$.
	
	Suppose $n\geq 1$, the Lipschitz function $f_{n}\colon \R^{d}\to\R$ and the collections $\left\{U_{i,l}\right\}_{i=1}^{\infty}$ of boxes with pairwise disjoint interiors are defined for $l=1,\ldots,n$ such that
	\begin{align}
		&\text{$f_{n}$ is nowhere differentiable in the set}\quad A_{n}\setminus\left(\bigcup_{l=1}^{n-1}\bigcup_{i=1}^{\infty}\partial U_{i,l}\right),\label{eq:fnnowhdiff}\\
		&E\cap\bigcup_{i=1}^{\infty}U_{i,n}=E\setminus A_{n}\quad\text{and}\quad\frac{\diam(U_{i,n})}{\dist(U_{i,n},A_{n})}\to 0\quad\text{as}\quad i\to\infty.\label{eq:nthcover}
	\end{align}
	Let the Lipschitz function $f_{n+1}\colon \R^{d}\to\R$ be given by the conclusion of Lemma~\ref{lemma:closed}, part~1 when we take $A=A_{n}$, $F=A_{n+1}$, $h=f_{n}$, $U_{i}=U_{i,n}$ and $\varepsilon=2^{-(n+1)}$. Then $f_{n+1}$ is nowhere differentiable in $(A_{n+1}\setminus A_{n})\setminus\bigcup_{i=1}^{\infty}\partial U_{i,n}$. From part 2 of Lemma~\ref{lemma:closed}, the differentiability of $f_{n+1}$ and $f_{n}$ coincides at all points of $A_{n}$. Hence, using \eqref{eq:fnnowhdiff}, $f_{n+1}$ is nowhere differentiable in the set $A_{n+1}\setminus\left(\bigcup_{l=1}^{n}\bigcup_{i=1}^{\infty}\partial{U_{i,n}}\right)$.
	
	Let the collection of boxes $\left\{U_{i,n+1}\right\}_{i=1}^{\infty}$ be given by the conclusion of Lemma~\ref{lemma:respectcover} when we take $F=E$, $A=A_{n+1}$ and $V_{i}=U_{i,n}$. This ensures the validity of \eqref{eq:nthcover} with $n$ replaced by $n+1$.
	
	We have defined, for each integer $n\geq 1$, a Lipschitz~function $f_{n}\colon \R^{d}\to\R$ and a collection of boxes $\left\{U_{i,n}\right\}_{i=1}^{\infty}$ with pairwise disjoint interiors. In addition to \eqref{eq:fnnowhdiff} and \eqref{eq:nthcover}. the construction ensures that the following conditions hold for each $n\geq 2$:

\begin{align}
\label{2k} \left\|f_{n}-f_{n-1}\right\|_{\infty}&\leq 2^{-n}\text{ and }\lip(f_{n})\leq\lip(f_{n-1})+2^{-n},\\
\label{4k} f_{n}(y)&=f_{n-1}(y)\text{ whenever }y\in\mathbb{R}^{d}\setminus\left(\bigcup_{i=1}^{\infty}\inter(U_{i,n-1})\right),\\
\label{5k}\text{For each index }&i\text{, there exists an index $j$ such that $U_{i,n}\subseteq U_{j,n-1}$.}
\end{align}
	For the sake of future reference we point out that
	\begin{equation}\label{eq:fixed}
	f_{m}(y)=f_{n}(y)\quad\text{whenever $m\geq n$ and}\quad y\in\mathbb{R}^{d}\setminus\left(\bigcup_{i=1}^{\infty}\inter(U_{i,n})\right).
	\end{equation}
	This follows from \eqref{4k} and \eqref{5k}. By \eqref{2k} the sequence $\left\{f_{n}\right\}_{n=1}^{\infty}$ converges uniformly to a Lipschitz function $f\colon\R^{d}\to\R$. Using \eqref{eq:fixed}, we deduce that the function $f$ satisfies
	\begin{equation}\label{eq:f=fn}
	f(y)=f_{n}(y)\quad\text{whenever}\quad y\in\R^{d}\setminus\left(\bigcup_{i=1}^{\infty}\inter (U_{i,n})\right).
	\end{equation}
	We are now ready to prove that $E$ is a non-universal differentiability set. In view of Lemma~\ref{lemma:removingboundaries}, it is sufficient to show that $E'=E\setminus\left(\bigcup_{k=1}^{\infty}\bigcup_{i=1}^{\infty}\partial U_{i,k}\right)$ is a non-universal differentiability set. We will prove that $f$ is nowhere differentiable in $E'$.
	
	Fix $x\in E'$ and choose $n$ such that $x\in A_{n}$. The condition \eqref{eq:nthcover} ensures that the conditions of Lemma~\ref{lemma:closed} are satisfied for $F=E$, $A=A_{n}$, $h=f_{n}$ and $U_{i}=U_{i,n}$. Further, from \eqref{eq:f=fn}, the hypothesis of Lemma~\ref{lemma:closed}, part~2 is satisfied for the function $f\colon \R^{d}\to\R$. Therefore, the differentiability of $f$ at $x$ coincides with that of $f_{n}$ at $x$ and, by \eqref{eq:fnnowhdiff}, the proof is complete.
\end{proof}

\section{Differentiability inside sets of positive measure.}
In this section we give an application of Theorem~\ref{theorem:fsigma} to differentiability inside sets of positive Lebesgue measure. 
\begin{theorem}\label{theorem:ptimesq}
	Let $d\geq 2$ and suppose $P_{1},P_{2},\ldots,P_{d}\subseteq\mathbb{R}$ are sets of positive one-dimensional Lebesgue measure. Then $P_{1}\times\ldots\times P_{d}$ contains a compact universal differentiability set with Lebesgue measure zero.
\end{theorem}
\begin{proof} We may assume that each set $P_{i}$ is closed. For $k=0,1,\ldots,d$, let $\Pi_{k}$ be the statement that $P_{1}\times P_{2}\times\ldots\times P_{k}\times\mathbb{R}^{d-k}$ contains a compact universal differentiability set $C_{k}$ with Lebesgue measure zero. The statement $\Pi_{0}$ is proved in \cite{doremaleva1}. Suppose now that $0<k\leq d$ and that the statement $\Pi_{k-1}$ holds. Let us prove the statement $\Pi_{k}$ and thus, by induction, Theorem~\ref{theorem:ptimesq}. 

	Let $\left\{r_{n}\right\}_{n=1}^{\infty}$ be a countable dense subset of $\mathbb{R}$ and consider the set
	\begin{equation*}
		F_{k}=\bigcup_{n=1}^{\infty}(\mathbb{R}^{k-1}\times (P_{k}+r_{n})\times\mathbb{R}^{d-k}).
	\end{equation*}
	Writing $F_{k,n}=\mathbb{R}^{k-1}\times (P_{k}+r_{n})\times\mathbb{R}^{d-k}$ for each $n$, we have $F_{k}=\bigcup_{n=1}^{\infty}F_{k,n}$ and each set $F_{k,n}$ is closed.
	Further, observe that $p_{k}(\mathbb{R}^{d}\setminus F_{k})$ is a subset of $\mathbb{R}$ with one-dimensional Lebesgue measure zero. We can write
	\begin{equation*}
		C_{k-1}=(C_{k-1}\cap F_{k})\cup (C_{k-1}\cap(\mathbb{R}^{d}\setminus F_{k})).
	\end{equation*}
	Since $\mathbb{R}^{d}\setminus F_{k}$ projects to a set of one-dimensional Lebesgue measure zero, we may apply \cite[Lemma~2.1]{dymondmaleva13} to conclude that $C_{k-1}\cap F_{k}$ is a universal differentiability set. Next, using Theorem~\ref{theorem:fsigma}, we deduce that there exists $n$ such that $C_{k-1}\cap F_{k,n}$ is a universal differentiability set. Setting 
	\begin{equation*}
		C_{k}=(C_{k-1}+-r_{n}e_{k})\cap (\mathbb{R}^{k-1}\times P_{k}\times \mathbb{R}^{d-k}),
	\end{equation*}
	we observe that
	\begin{equation*}
		C_{k}=(C_{k-1}\cap F_{k,n})-r_{n}e_{k}.
	\end{equation*}
	$C_{k}$ is a universial differentiability set, due to the easily verified fact that any translate of a universal differentiability set is a universal differentiability set. Note that $(C_{k-1}+\lambda_{n})\subseteq P_{1}\times\ldots\times P_{k-1}\times \mathbb{R}^{d-k+1}$. Hence, $C_{k}\subseteq P_{1}\times \ldots\times P_{k}\times \mathbb{R}^{d-k}$ and the proof of statement $\Pi_{k}$ is complete.
\end{proof}
The above Theorem~\ref{theorem:ptimesq} provides a partial answer to the following question of Godefroy: Does every subset of $\mathbb{R}^{d}$ with positive Lebesgue measure contain a universal differentiability set of Lebesgue measure zero? This question was asked following a talk of Maleva at the 2012 conference `Geometry of Banach spaces' in CIRM, Luminy, and remains open. Theorem~\ref{theorem:ptimesq} also builds on an observation of Dor\'e and Maleva: A consequence of Lemma~3.5 in~\cite{doremaleva3} is that every set of the form $P\times\mathbb{R}^{d-1}\subseteq\mathbb{R}^{d}$, where $P\subseteq\mathbb{R}$ is a set of positive Lebesgue measure, contains a Lebesgue null universal differentiability set.

\textbf{Acknowledgement.} The author wishes to thank Olga Maleva for helpful discussions.

\bibliographystyle{plain}
\bibliography{biblio}   
Michael Dymond\\
Institut für Mathematik\\
Universit\"at Innsbruck\\
Technikerstraße 13,
6020 Innsbruck,
\"Osterreich (Austria)\\
\texttt{michael.dymond@uibk.ac.at}\\[3mm]

\end{document}